\documentclass[12pt]{article}

\usepackage{graphicx,amssymb,eucal,mathrsfs}
\usepackage{latexsym,amsmath,epsfig,epic,eepic}
\usepackage{amscd}
\usepackage{amsthm}
\usepackage{indentfirst}
\usepackage{epsf,graphicx,pgf}
\usepackage{pstricks,mathrsfs}

\newtheorem{theorem}{{\bf Theorem}}[section]

\newtheorem{prop}[theorem]{{\bf Proposition}}
\newtheorem{proposition}[theorem]{{\bf Proposition}}
\newtheorem{lemma}[theorem]{{\bf Lemma}}
\newtheorem{question}[theorem]{{\bf Question}}
\newtheorem{corollary}[theorem]{{\bf Corollary}}
\newtheorem{remark}[theorem]{{\bf Remark}}

\newtheorem{exam}[theorem]{Example}

\newcommand{\A}{\mathbb{A}}

\newcommand{\R}{\mathbb{\bf R}}
\newcommand{\Z}{\mathbb{\bf Z}}

\newcommand{\D}{\mathbb{D}}

\newcommand{\identity}{1\hskip-2.5pt{\rm l}}
\newcommand{\N}{\mathbb{\bf N}}
\newcommand{\T}{\mathbb{T}^{2}}
\newcommand{\ham}{\mathrm{Ham}}

\newcommand{\SPH}{\mathbb{S}}

\newcommand{\cD}{{\mathcal{D}}}
\newcommand{\cH}{{\mathcal{H}}}

\newcommand{\cU}{{\mathcal{U}}}
\newcommand{\cV}{{\mathcal{V}}}

\newcommand{\dif}{\mathrm{Diff}_{0,c}}

\newcommand{\cO}{{\mathcal{O}}}

\newcommand{\tK}{{\tilde{K}}}
\newcommand{\tpsi}{{\tilde{\psi}}}
\newcommand{\Qed}{\hfill \qedsymbol \medskip}

\begin{document}

\title{On continuity of quasi-morphisms for symplectic maps \\
{\it \normalsize Dedicated to Oleg Viro on the occasion of his 60th birthday}}

\renewcommand{\thefootnote}{\alph{footnote}}

\author{\textsc Michael Entov$^{a}$,\ Leonid
Polterovich$^{b}$, Pierre Py\\ \\with an appendix by Michael
Khanevsky}

\footnotetext[1]{Partially supported by the Israel Science
Foundation grant $\#$ 881/06.} \footnotetext[2]{Partially
supported by the Israel Science Foundation grant $\#$ 509/07.}

\date{April 2009}


\maketitle

\begin{abstract} We discuss $C^0$-continuous homogeneous quasi-morphisms on
 the identity component of the group of compactly supported symplectomorphisms
 of a symplectic manifold.
Such quasi-morphisms extend to the  $C^0$-closure of this group
inside the homeomorphism group. We show that for standard
symplectic balls of any dimension, as well as for compact oriented
surfaces, other than
 the sphere, the space of such quasi-morphisms is infinite-dimensional.
 In the case of surfaces, we give a user-friendly
topological characterization of such quasi-morphisms. We also
present an application to Hofer's geometry on the group of
Hamiltonian diffeomorphisms of the ball.
\end{abstract}

\tableofcontents

\section{Introduction and main results}

\subsection{Quasi-morphisms on groups of symplectic maps}\label{subseq-aph}

Let $(\Sigma,\omega)$ be a compact connected symplectic manifold (possibly with a
non-empty boundary $\partial \Sigma$). Denote by $\cD(\Sigma,\omega)$ the identity
component of the group of symplectic $C^{\infty}$-diffeomorphisms
of $\Sigma$ whose supports lie in the interior of $\Sigma$. Write\footnote{We abbreviate $\cD(\Sigma)$ and $\cH(\Sigma)$ whenever
the symplectic form $\omega$ is clear from the context.}
$\cH(\Sigma,\omega)$ for the $C^0$-closure of $\cD(\Sigma,\omega)$
in the group of homeomorphisms of $\Sigma$ supported in the
interior of $\Sigma$. We always equip $\Sigma$ with a distance $d$ induced by a Riemannian metric on $\Sigma$, and view
the $C^0$-topology on the group of homeomorphisms of $\Sigma$ as
the topology defined by the metric ${\it dist} (\phi,\psi) =
\max_{x\in\Sigma} d(x, \psi^{-1}\phi(x))$.

The study of the algebraic structure of the groups
$\cD(\Sigma,\omega)$ was pioneered by Banyaga, see
\cite{Ban78,banyaga-book}. For instance, when $\Sigma$ is closed, he calculated the
commutator subgroup of $\cD(\Sigma,\omega)$ and showed that it is
simple. However, the algebraic structure of the groups
$\cH(\Sigma,\omega)$ is much less understood. Even for the
standard two-dimensional disc $\D^2$ it is still unknown whether
$\cH(\D^2)$ coincides with its commutator subgroup or not (see
\cite{bounemoura} for a comprehensive discussion). In the present
paper we focus on the following algebraic feature of the groups
$\cH(\Sigma,\omega)$.

Recall that a homogeneous quasi-morphism on a group $\Gamma$ is a
map $ \mu : \Gamma \to \R$ which satisfies the following two
properties:

\medskip
\noindent
\begin{itemize}
\item[{(i)}] There exists a constant $C(\mu) \ge 0$ such that
$\vert \mu(xy)-\mu(x)-\mu(y)\vert \le C(\mu)$ for any $x$, $y$ in
$\Gamma$. \item[{(ii)}] $\mu(x^{n})=n\mu(x)$ for any $x\in \Gamma$
and $n\in \Z$.
\end{itemize}

\medskip
\noindent Let us recall two well-known properties of homogeneous
quasi-morphisms which will be useful in the sequel: they are
invariant under conjugation, and their restrictions to abelian
subgroups are homomorphisms.

The space of all homogeneous quasi-morphisms is an important
algebraic invariant of the group. Quasi-morphisms naturally appear
in the theory of bounded cohomology and are crucial in the study
of the commutator length \cite{bav}. We refer to \cite{bav}, \cite{calegari}, \cite{ghys} or
\cite{kotschick-what} for a more detailed introduction to the
theory of quasi-morphisms.

Recently, several authors discovered that certain groups of
diffeomorphisms preserving a volume or a symplectic form carry
homogeneous quasi-morphisms, see
\cite{Barge-Ghys,bensimon,Entov,entpol,EP-toric,gg,ostrover,Polt-Montreal,py1}.
However, in many cases explicit constructions of non-trivial
quasi-morphisms on $\cD(\Sigma,\omega)$ require certain smoothness
in an essential manner. Nevertheless, as we shall show below, some
homogeneous quasi-morphisms can be extended from
$\cD(\Sigma,\omega)$ to $\cH(\Sigma,\omega)$.

Our first result deals with the case of the Euclidean unit ball
$\D^{2n}$ in the standard symplectic linear space.

\bigskip
\noindent
\begin{theorem} \label{theorem-highdim} The space of homogeneous quasi-morphisms
on $\cH(\D^{2n})$ is infinite-dimensional.
\end{theorem}

\medskip
\noindent The proof is given in Section~\ref{sec-highdim}. Next,
we focus on the case when $\Sigma$ is a compact connected surface
equipped with an area form. Note that in this case $\cH(\Sigma)$
coincides with the identity component of the group of all
area-preserving homeomorphisms supported in the interior of
$\Sigma$, see \cite{oh} or \cite{sikorav}.

\medskip
\noindent
\begin{theorem} \label{theorem-main-1} Let $\Sigma$ be a compact
connected oriented surface other than the sphere $\SPH^2$,
equipped with an area form. The space of homogeneous
quasi-morphisms on $\cH(\Sigma)$ is infinite-dimensional.
\end{theorem}

\medskip
\noindent The proof is given in Section~\ref{examples}. This
result is new, for instance, in the case of the 2-torus. The case
of the sphere is still out of reach -- see
Section~\ref{sect-sphere} for a discussion. Interestingly enough,
for balls of any dimension and for the two-dimensional annulus,
all our examples of homogeneous quasi-morphisms on $\cH$ are based
on Floer theory. When $\Sigma$ is of genus greater than one, the
group $\cH(\Sigma)$ carries a lot of homogeneous quasi-morphisms,
and the statement of Theorem~\ref{theorem-main-1}
readily follows from the work of Gambaudo and Ghys \cite{gg}.

As an immediate application, Theorems \ref{theorem-highdim} and
\ref{theorem-main-1} yield that if $\Sigma$ is a ball or a closed
oriented surface other than sphere, then the stable commutator
length is unbounded on the commutator subgroup of $\cH(\Sigma)$.
This is a standard consequence of Bavard's theory \cite{bav}.

\subsection{Detecting continuity}\label{subsec-det}

A key ingredient of our approach is the following proposition, due
to Shtern \cite{Shtern}. It is a simple (nonlinear)
analogue of the fact that linear forms on a topological vector
space are continuous if and only if they are bounded in a
neighborhood of the origin.

\begin{proposition}[\cite{Shtern}]
\label{prop-general-top-group} Let $\Gamma$ be a topological group
and $\mu: \Gamma\to\R$ a homogeneous quasi-morphism. Then $\mu$ is
continuous if and only if it is bounded on a neighborhood of the
identity.
\end{proposition}

\begin{proof}
We only prove the ``if" part. Assume that $|\phi |$ is bounded by $K>0$
on an open neighborhood $\cU$ of the identity. Let $g \in \Gamma$. For each $p \in\N$ define
$$\cV_p(g) := \{ h\in \Gamma\ | \ h^p\in g^p \cU\}.$$
It is easy to see that $\cV_p(g)$ is an open neighborhood of $g$.
Pick any $h\in \cV_p(g)$. Then $h^p = g^p f$ for some $f\in \cU$.
Therefore
$$|\phi (h^p) - \phi (g^p) - \phi (f)|\leq C(\phi),$$
hence
$$|\phi (h) - \phi (g)| \leq \frac{C(\phi) + K}{p},$$
which immediately yields the continuity of $\phi$ at $g$.
\end{proof}

Let us discuss in more details the extension problem for
quasi-morphisms. The next proposition shows that $C^0$-{\bf
continuous} homogeneous quasi-mor\-phisms on $\cD(\Sigma)$ extend
to $\cH(\Sigma)$.

\medskip
\noindent
\begin{proposition}
\label{prop-general-top-group-1} Let $\Lambda$ be a topological
group and let $\Gamma \subset \Lambda$ be a dense subgroup. Any
{\bf continuous} homogeneous quasi-morphism on $\Gamma$ extends to
a continuous homogeneous quasi-morphism on $\Lambda$.
\end{proposition}

\begin{proof}  Since $\mu$ is continuous, it is bounded by a
constant $C>0$ on an open neighborhood $\cU$ of the identity in
$\Gamma$. Since $\cU$ is open in $\Gamma$, there exists $\cU'$,
open in $\Lambda$, so that $\cU=\cU'\cap \Gamma$. We fix an open
neighborhood $\cO$ of the identity in $\Lambda$ so that
$\cO^{2}\subset \cU'$ and $\cO=\cO^{-1}$. Given $g\in \Lambda$ and
$p\in\N$, define as before:
$$\cV_p (g) := \{ h\in \Lambda\ | \ h^p\in g^p \cO \}.$$
Pick a sequence
$\{h_k\}$ in $\Gamma$ so that each $h_k$ lies in $\cV_1
(g)\cap\ldots\cap \cV_k (g)$. For $k\ge p$ we can write
$h_{k}^{p}=g^{p}g_{k,p}$ ($g_{k,p}\in \cO$). If $k_{1},k_{2}\ge
p$, we can write
$$h_{k_{1}}^{p}=h_{k_{2}}^{p}g_{k_{2},p}^{-1}g_{k_{1},p}, \; \; \; \; \; g_{k_{2},p}^{-1}g_{k_{1},p}\in \cU.$$
Hence, we have the inequality
$$\left| \mu(h_{k_{1}})-\mu(h_{k_{2}})\right| \le \frac{C+C(\mu)}{p} \;\;\;\;\; (k_{1}, k_{2}\ge p),$$
and $\{ \mu (h_p)\}$ is a Cauchy sequence in $\R$. Denote its
limit by $\mu' (g)$. One can check easily that the definition is
correct and that for any sequence $g_i \in \Gamma$ converging to
$g \in \Lambda$ one has $\mu(g_i) \to \mu'(g)$. This readily
yields that the resulting function $\mu': \Lambda\to\R$ is a
homogeneous quasi-morphism extending $\mu$. Its continuity follows
from Proposition~\ref{prop-general-top-group}.\end{proof}

\medskip
\noindent In view of this proposition all we need for the proof of
Theorems~\ref{theorem-highdim} and  \ref{theorem-main-1} is to
exhibit non-trivial homogeneous quasi-morphisms on $\cD(\Sigma)$
which are continuous in the $C^0$-topology. This leads us to the
problem of continuity of homogeneous quasi-morphisms which is
highlighted in the title of the present paper.

\begin{remark} Note that all the concrete quasi-morphisms that we know on groups of diffeomorphisms are continuous in the $C^{1}$-topology. 
\end{remark}

\subsection{Calabi homomorphism and continuity on surfaces}
\label{sec-calabi-homom-discontinuous}

It is a classical fact that the Calabi homomorphism
 is not continuous in the
$C^0$-topology, see \cite{gg0}. We will discuss the example of the unit ball in $\R^{2n}$ and
then explain why the reason for the discontinuity of the Calabi
homomorphism is, in a sense, universal.

First, let us recall the definition of the group of Hamiltonian
diffeomorphisms of a symplectic manifold $(\Sigma,\omega)$.
Given a smooth function $F: \Sigma \times S^{1} \to \R$, supported
in $\text{Interior}(\Sigma) \times S^1$, consider the
time-dependent vector field $\text{sgrad} F_t$ given by
$i_{\text{sgrad} F_t} = -dF_t$, where $F_t(x)$ stands for
$F(x,t)$. The flow $f_t$ of this vector field is called {\it the
Hamiltonian flow generated by the Hamiltonian function} $F$ and
its time one map $f_1$ is called {\it the Hamiltonian
diffeomorphism generated by} $F$. Hamiltonian diffeomorphisms form
a normal subgroup of $\cD(\Sigma,\omega)$ denoted by
$\ham(\Sigma,\omega)$, or just by $\ham(\Sigma)$. The quotient $\cD(\Sigma)/\ham(\Sigma)$ is isomorphic to the group
$H^1_{comp}(\Sigma,\R)$. In particular, $\cD(\Sigma)=\ham
(\Sigma)$ for $\Sigma = \D^{2n}$ or for $\Sigma = \SPH^2$. We
refer to \cite{McD-Sal} for the details.

\medskip
\noindent
\begin{exam}\label{ex-cal-homo} {\rm
Let $\Sigma= \D^{2n}$ be the closed unit ball in $\R^{2n}$
equipped with the symplectic form $\omega = dp\wedge dq$. Take any
diffeomorphism $f \in \ham(\D^{2n})$ and pick a Hamiltonian $F$ generating $f$.
The value
$${\it Cal}(f):= \int_0^1 \int_{\D^{2n}} F(p,q,t) \; dpdq\;dt$$
depends only on $f$ and defines the {\it Calabi homomorphism} ${\it
Cal}:\cD(\D^{2n}) \to \R$ \cite{Calabi}.

Take a sequence of time-independent Hamiltonians $F_i$ supported
in balls of radii $\frac{1}{i}$ so that $\int_{\D^{2n}} F_i dp dq
= 1$. The corresponding Hamiltonian diffeomorphisms $f_i$
$C^0$-converge to the identity and satisfy ${\it Cal}(f_i)=1$. We
conclude that the Calabi homomorphism is discontinuous in the
$C^0$-topology.}
\end{exam}

In the remainder of this section, let us return to the case when
$\Sigma$ is a compact connected surface equipped with an area
form. Our next
result shows, roughly speaking, that for a quasi-morphism $\mu$ on
$\ham (\Sigma)$ its non-vanishing on a sequence of Hamiltonian
diffeomorphisms $f_i$ supported in a collection of shrinking balls
is the only possible reason for discontinuity. The next remark is
crucial for understanding this phenomenon. Observe that
$\text{support}(f^N) \subset \text{support}(f)$ for any
diffeomorphism $f$. Thus in the statement above non-vanishing
yields unboundedness: if $\mu(f_i) \neq 0$ for all $i$ then the
sequence $\mu(f_i^{N_i})= N_i\mu(f_i)$ is unbounded for an
appropriate choice of $N_i$.

\medskip
\noindent
\begin{theorem}\label{criterion}
Let $\mu : \ham(\Sigma)\to \R$ be a homogeneous quasi-morphism.
Then $\mu$ is continuous in the $C^{0}$-topology if and only if
there exists $a>0$ such that the following property holds: For any
disc $D\subset \Sigma$ of area less than $a$ the restriction of
$\mu$ to the group $\ham (D)$ vanishes.
\end{theorem}

\medskip
\noindent Here by a disc in $\Sigma$ we mean the image of a smooth
embedding $ \D^2 \hookrightarrow \Sigma$. We view it as a surface
with boundary equipped with the area form which is the restriction
of the area form on $\Sigma$. The ``only if" part of the theorem
is elementary. It extends to certain four-dimensional symplectic
manifolds (see Remark~\ref{rem-inj} below). The proof of the ``if"
part is more involved and no extension to higher dimensions is
available to us so far (see the discussion in
Section~\ref{subsec-hdpro} below).

\medskip
\noindent
\begin{corollary}\label{criterion-1}Let $\mu : \cD(\Sigma)\to \R$ be a homogeneous quasi-morphism.
Suppose that
\begin{itemize}
\item[{(i)}] There exists $a>0$ such that for any disc $D\subset
\Sigma$ of area less than $a$ the restriction of $\mu$ to the
group $\ham (D)$ vanishes. \item[{(ii)}] The restriction of $\mu$
to each one-parameter subgroup of $\cD (\Sigma)$ is linear.
\end{itemize}
Then $\mu$ is continuous in the $C^{0}$-topology.
\end{corollary}

\medskip
\noindent Note that assumption (ii) is indeed necessary, provided
one believes in the axiom of choice. Indeed, assuming that
$\Sigma$ is not $\D^2,\SPH^2$ or $\T$, the quotient
$\cD(\Sigma)/\ham(\Sigma)$ is isomorphic to the additive group of
a vector space $V:=H^1_{comp}(\Sigma,\R)\neq \{0\}$. Define a
quasi-morphism $\mu: \cD(\Sigma) \to \R$ as the composition of the
projection $\cD(\Sigma) \to V$ with a {\it discontinuous}
homomorphism $V \to \R$. The homomorphism $\mu$ satisfies (i),
since it vanishes on $\ham(\Sigma)$, and it is obviously
discontinuous.

The  criterion  of continuity stated in Theorem~\ref{criterion}
and Corollary~\ref{criterion-1} are proved in
Section~\ref{sec-main-result-pf}. They will be used in
Section~\ref{examples} in order to verify $C^0$-continuity of a
certain family of quasi-morphisms on $\cD(\T)$ introduced in
\cite{gg} and explored in \cite{py}, which will enable us to
complete the proof of Theorem~\ref{theorem-main-1}.

\subsection{An application to Hofer's geometry}\label{subsec-app}

Here we concentrate on the case of the unit ball $\D^{2n} \subset
\R^{2n}$. For a diffeomorphism $f \in \ham(\D^{2n})$ define its
Hofer's norm \cite{Hofer90} as
$$||f||_{H}:= \inf \int_0^1 \left(\max_{z \in \D^{2n}} F(z,t) - \min_{z \in \D^{2n}} F(z,t) \right)\;
dt\;,$$ where the infimum is taken over all Hamiltonian functions
$F$ generating $f$. Hofer's famous result states that $d_H(f,g):=
||fg^{-1}||_H$ is a non-degenerate bi-invariant metric on
$\ham(\D^{2n})$. It is called {\it Hofer's metric.} It turns out
that the quasi-morphisms that we construct in the proof of
Theorem~\ref{theorem-highdim} are Lipschitz with respect to
Hofer's metric. Hence, our proof of Theorem~\ref{theorem-highdim}
yields:

\begin{proposition} The space of homogeneous quasi-morphisms on the
group $\ham(\D^{2n})$ which are both continuous for the
$C^{0}$-topology and Lipschitz for Hofer's metric is infinite
dimensional.
\end{proposition}

The relation between Hofer's metric and the $C^0$-metric on $\ham
(\Sigma)$ is subtle. First of all, the $C^0$-metric is never
continuous with respect to Hofer's metric. Furthermore, arguing as
in Example \ref{ex-cal-homo} one can show that Hofer's metric on
$\ham(\D^{2n})$ is not continuous in the $C^0$-topology. However,
for $\R^{2n}$ equipped with the standard symplectic form $dp
\wedge dq$ (informally speaking, this corresponds to the case of
ball of infinite radius), Hofer's metric is continuous for the $C^{0}$-Whitney topology
\cite{Hofer93}.

An attempt to understand the relation between Hofer's metric and
the $C^0$-metric led Le Roux \cite{leroux} to the following
problem. Let $\mathscr{E}_{C} \subset \ham(\D^{2n})$ be the
complement of the closed ball (in Hofer's metric) of radius $C$
centered at the identity:

$$\mathscr{E}_{C}:=\{f\in \ham(\D^{2n}), d_{H}(f,\identity)>C\}.$$

\noindent Le Roux asked the following: {\it is it true that $\mathscr{E}_{C}$ has a
non-empty interior in the $C^{0}$-topology for any $C>0$?}

The energy-capacity inequality \cite{Hofer90} states that if $f
\in \ham(\D^{2n})$ displaces $\phi(\D^{2n}(r))$, where $\phi$ is
any symplectic embedding of the Euclidean ball of radius $r$, then
Hofer's norm of $f$ is at least $\pi r^2$. (We say that $f$
{\it displaces a set $U$}, if $f(U)\cap \bar{U}=\emptyset$). By
Gromov's packing inequality \cite{Gromov}, this could happen only
when $r^2 \leq 1/2$. Since the displacement property is
$C^0$-robust, we get that $\mathscr{E}_C$ indeed has a non-empty
interior in the $C^0$-sense for  $C < \pi/2$. Using our
quasi-morphisms we get an affirmative answer to
Le Roux's question even for large values of $C$.

\begin{corollary}\label{topology}

For any $C>0$ the set $\mathscr{E}_{C}$ has a non-empty interior in
the $C^{0}$-topology.

\end{corollary}

\begin{proof}
This follows just from the existence of a nontrivial homogeneous quasi-morphism
$\mu : \ham(\D^{2n})\to \R$ which is both continuous in the
$C^{0}$-topology and Lipschitz with respect to Hofer's metric.
Indeed, pick a diffeomorphism $f$ such that $$\frac{\vert
\mu(f)\vert}{{\rm Lip}(\mu)}\ge C+1,$$ where ${\rm Lip} (\mu)$ is
the Lipschitz constant of $\mu$ with respect to Hofer's metric.
There is a neighborhood $O$ of $f$ in $\ham(\D^{2n})$ in the
$C^{0}$-topology on which $\vert \mu \vert > C\cdot {\rm Lip}(\mu)$.
We get that $||g||_{H}>C$ for $g\in O$ and hence
$O\subset \mathscr{E}_{C}$. This proves the corollary.
\end{proof}

Note that Le Roux's question makes sense on any symplectic
manifold. For certain closed symplectic manifolds with infinite
fundamental group one can easily get a positive answer using the
energy-capacity inequality
 in the universal cover (as in \cite{lalondepol}).
However, for closed simply connected manifolds (and already for
the case of the $2$-sphere) the question is wide open.

\subsection{Acknowledgements}
This text started as an attempt to understand a remark of Dieter
Kot\-schick. We thank him for stimulating discussions and in particular for communicating to us the idea of getting the continuity from the $C^{0}$-fragmentation, which appeared in a preliminary version of \cite{kotschick}. The authors
would like to thank warmly Fr\'ed\'eric Le Roux for
 his comments on this work and for the thrilling discussions we had
  during the preparation of this article. The third author would like to thank
   Tel-Aviv University for its hospitality during the spring of 2008, when this work started. The second author expresses his deep gratitude to Oleg Viro for a generous help and support at the beginning of his research in topology.

\section{Quasi-morphisms for the ball}\label{sec-highdim}

In this section we prove Theorem~\ref{theorem-highdim}.

Denote by $\D^{2n}(r)$ the  Euclidean ball $\{|p|^2 +|q|^2 \leq
r^2\}$, so that $\D^{2n}= \D^{2n}(1)$. We say that a set $U$ in a
symplectic manifold $(\Sigma,\omega)$ is {\it displaceable}, if
there exists $\phi\in \ham(\Sigma)$ which displaces it:
$\phi(U)\cap \bar{U}=\emptyset$. A quasi-morphism $\mu: \ham
(\Sigma)\to\R$ will be called {\it Calabi}, if for any
displaceable domain $U\subset M$, such that $\left.
\omega\right|_U$ is exact, one has $\left. \mu\right|_{\ham(U)} =
\left. {\it Cal}\right|_{\ham(U)}$.

We will use the following result, established in \cite{entpol}: there exists $a>0$ so that the group $\ham(\D^{2n}(1+a))$ admits
an infinite-dimensional space of quasi-morphisms which are
Lipschitz in Hofer's metric, vanish on $\ham (U)$ for every
displaceable domain $U \subset \D^{2n}(1+a)$ and do not vanish on
$\ham(\D^{2n})$.
 These
quasi-morphisms are obtained by subtracting the appropriate
multiple of the Calabi homomorphism from the Calabi
quasi-morphisms constructed in \cite{BEP}. We claim that the
restriction of each such quasi-morphism, say $\eta$, to
$\ham(\D^{2n})$ is continuous in the $C^0$-topology. By
Proposition~\ref{prop-general-top-group-1}, this would yield the
desired result. By Proposition~\ref{prop-general-top-group} it
suffices to show that for some $\epsilon >0$ the quasi-morphism
$\eta$ is bounded on all $f \in \ham (\D^{2n})$ such that
\begin{equation}\label{eq-hdm}
|f(x)-x| < \epsilon \;\; \forall x \in \D^{2n}\;.
\end{equation}
For $c >0$ define the strip
$$\Pi(c):= \{(p,q) \in \R^{2n}\;:\; |q_n| < c\}\;.$$
Choose $\epsilon >0$ so small that $\Pi(2 \epsilon) \cap \D^{2n}$
is displaceable in $\D^{2n}(1+a)$. Put $D_{\pm} := \D^{2n} \cap
\{\pm q_n
>0\}$. Observe that $D_{\pm}$ are displaceable in $\D^{2n}(1+a)$ by a
Hamiltonian diffeomorphism that can be represented outside a
neighborhood of the boundary as a small vertical shift along the
$q_n$-axis (in the case of $D_+$ we take the shift that moves it
up and in the case of $D_-$ the shift that moves it down) composed
with a $180$ degrees rotation in the $(p_n,q_n)$-coordinate plane.
The desired boundedness result immediately follows from the next
fragmentation-type lemma:

\medskip
\noindent
\begin{lemma}
\label{lem-frag-hdim} Assume that $f \in \ham(\D^{2n})$ satisfies
\eqref{eq-hdm}. Then $f$ can be decomposed as $\theta
\phi_+\phi_-$, where $\theta \in \ham(\Pi(2 \epsilon) \cap
\D^{2n})$ and $\phi_{\pm} \in \ham(D_{\pm})$.
\end{lemma}

\medskip
\noindent Indeed, $\eta$ vanishes on $\ham (U)$ for every
displaceable domain $U \subset \D^{2n}(1+a)$. Since $\Pi(2
\epsilon) \cap \D^{2n}$ and $D_{\pm}$ are displaceable,
$\eta(\theta)=\eta(\phi_{\pm})=0$. Thus  $|\eta(f)| \leq 2C(\eta)$
for every $f\in \ham(\D^{2n})$ lying in the
$\epsilon$-neighborhood of the identity with respect to the
$C^0$-distance, and the theorem follows. It remains to prove the
lemma.

\medskip
\noindent {\bf Proof of Lemma~\ref{lem-frag-hdim}:} Denote by $S$
the hyperplane $\{q_n=0\}$. For $c>0$ write $R_c$ for the dilation
$z \to cz$ of $\R^{2n}$. We assume that all compactly
supported diffeomorphisms of $\D^{2n}$ are extended to the whole
$\R^{2n}$ by the identity.

Take $f \in \ham(\D^{2n})$ which satisfies \eqref{eq-hdm}. Let
$\{f_t\}_{0\le t \le 1}$, be a Hamiltonian isotopy supported in
$\D^{2n}$ so that $f_t = \identity$ for $t \in [0,\delta)$ and
$f_t = f$ for $t \in (1-\delta,1]$ for some $\delta > 0$. Take
a smooth function $c: [0,1] \to [1, +\infty)$ which equals $1$
near $0$ and $1$ and satisfies $c(t) > (2\epsilon)^{-1}$ on
$[\delta,1-\delta]$. Consider the Hamiltonian isotopy $h_t =
R_{1/c(t)}f_tR_{c(t)}$ of $\R^{2n}$. Note that $h_0=\identity$ and
$h_1=f$. Since $c(t)\geq 1$, we have $h_tz =z$ for $z \notin
\D^{2n}$, and $h_t$ is supported in $\D^{2n}$.

We claim that $h_t(S) \subset \Pi(2\epsilon)$. Observe that
$R_{c(t)}S =S$. Take any $z \in S$. If $R_{c(t)}z \notin \D^{2n}$,
we have that $h_tz =z$. Assume now that $R_{c(t)}z \in \D^{2n}$.
Consider the following cases:

\begin{itemize}

\item{} If $t \in (1-\delta,1]$, then $f_t R_{c(t)}(S) = f(S)$.
Thus $f_tR_{c(t)}z \in f(S \cap \D^{2n}) \subset \Pi(2\epsilon)$,
where the latter inclusion follows from \eqref{eq-hdm}. Therefore
$h_tz \in \Pi(2\epsilon)$ since $c(t) \geq 1$.

\item{} If $t \in [\delta,1-\delta]$, then $h_tz \in
\D^{2n}(2\epsilon) \subset \Pi(2\epsilon)$ by our choice of the
function $c(t)$.

\item{} If $t \in [0,\delta)$, then $h_t S = S \subset
\Pi(2\epsilon)$.

\end{itemize}

This completes the proof of the claim. 

By continuity of $h_t$, there exists $\kappa >0$ so that
$h_t(\Pi(\kappa)) \subset \Pi(2\epsilon)$ for all $t$. Cutting off
the Hamiltonian of $h_t$ near $h_t(\Pi(\kappa))$ we get a
Hamiltonian flow $\theta_t$ supported in $\Pi(2\epsilon)$ which
coincides with $h_t$ on $\Pi(\kappa)$. Thus, $\theta_t^{-1}h_t$ is
the identity on $\Pi(\kappa)$ for all $t$. It follows that
$\theta_t^{-1}h_t$ decomposes into the product of two commuting
Hamiltonian flows $\phi_-^t$ and $\phi_+^t$ supported in $D_-$ and
$D_+$ respectively. Therefore $f = \theta_1 \phi_-^1 \phi_+^1$ is
the desired decomposition. \qed

\section{Proof of the criterion of continuity on surfaces}
\label{sec-main-result-pf}

\subsection{A $C^0$-small fragmentation theorem on surfaces}
Before stating our next result we recall the notion of {\it
fragmentation} of a diffeomorphism. This is a classical technique
in the study of groups of diffeomorphisms, see e.g.
\cite{Ban78,banyaga-book,bounemoura}. Given a Hamiltonian
diffeomorphism $f$ of a connected symplectic manifold $\Sigma$,
and an open cover $\{U_{\alpha}\}$ of $\Sigma$, one can always
write $f$ as a product of Hamiltonian diffeomorphisms each of
which is supported in one of the open sets $U_{\alpha}$. It is
known that the number of factors in such a decomposition is
uniform in a $C^{1}$-neighborhood of the identity, see
\cite{Ban78, banyaga-book,bounemoura}. To prove our
continuity theorem we actually need to prove a similar result on
surfaces when one consider diffeomorphisms endowed with the
$C^{0}$-topology. Such a result appears in \cite{leroux2} in the
case when the surface is the unit disc. Observe also that the
corresponding fragmentation result is known for volume-preserving
homeomorphisms \cite{fathi}.

\medskip
\noindent
\begin{theorem}
\label{lemma-N0-small-discs} Let $\Sigma$ be a compact connected
surface (possibly with boundary), equipped with an area-form. Then
for every $a>0$ there exists a neighborhood $\cU$ of the identity
in the group $\ham(\Sigma)$ endowed with the $C^{0}$-topology and
an integer $N>0$ such that any diffeomorphism $g\in \cU$ can be
written as a product of at most $N$ Hamiltonian diffeomorphisms
supported in discs of area less than $a$.
\end{theorem}

\medskip
\noindent This result might be well-known to experts and probably
can be deduced from the corresponding result for homeomorphisms.
 However, since the proof is more difficult for homeomorphisms,
 and in order to keep this paper self-contained, we are going to give a
 direct proof of Theorem~\ref{lemma-N0-small-discs} in
 Section~\ref{sec-fragmentation}. Note that this last section is the most technical part of the text. Given the fragmentation result above,
 one obtains easily a proof of Theorem~\ref{criterion}, as we will show now.

\subsection{Proof of Theorem~\ref{criterion} and Corollary~\ref{criterion-1}} 1)We begin by proving that the
condition appearing in the statement of the theorem is necessary
for the quasi-morphism $\mu$ to be continuous. Assume $\mu$ is
continuous for the $C^{0}$-topology. Then it is bounded on some
$C^0$-neighborhood $\cU$ of the identity in $\ham(\Sigma)$. Choose
now a disc $D_{0}$ in $\Sigma$. If $D_{0}$ has a sufficiently
small diameter, then $\ham(D_{0})\subset \cU$. But since
$\ham(D_{0})$ is a subgroup and $\mu$ is homogeneous, $\mu$ must
vanish on $\ham(D_{0})$.

Now, let $a={\rm area}(D_{0})$. If $D$ is any disc of area less
than $a$, the group $\ham(D)$ is conjugated in $\ham(\Sigma)$ to a
subgroup of $\ham(D_{0})$, because for any two discs of the same
area in $\Sigma$ there exists a Hamiltonian diffeomorphism mapping
one of the discs onto another -- see e.g. \cite{Akveld-Salamon},
Proposition A.1, for a proof (which, in fact, works for all
$\Sigma$, though the claim there is stated only for closed
surfaces). Hence, $\mu$ vanishes on $\ham(D)$ as required.

\medskip
\noindent
\begin{remark}\label{rem-inj} {\rm This proof extends verbatim to
higher-dimensional symplectic manifolds $(\Sigma,\omega)$ which
admit a positive constant $a_0$ with the following property: for
every $a < a_0$  all symplectically embedded balls of volume $a$
in the interior of $\Sigma$ are Hamiltonian isotopic. Here a
symplectically embedded ball of volume $a$ is the image of the
standard Euclidean ball of volume $a$ in $(\R^{2n}, dp \wedge dq)$
under a symplectic embedding. This property holds for instance for
blow-ups of rational and ruled symplectic four-manifolds,
see \cite{McDuff-1,Lalonde,Biran,McDuff-2}.}
\end{remark}

\medskip

2) We now prove the reverse implication. Assume that a homogeneous
quasi-morphism $\mu$ vanishes on all Hamiltonian diffeomorphisms
supported in discs of area $< a$. Take the $C^0$-neighborhood
$\cU$ of the identity and the integer $N$  from Theorem
\ref{lemma-N0-small-discs}. Then $\mu$ is bounded by $(N-1)
C(\mu)$ on $\cU$, and hence, continuous by
Proposition~\ref{prop-general-top-group}. \qed

\medskip

We now prove Corollary~\ref{criterion-1}. Choose compactly
supported symplectic vector fields $v_1,\ldots,v_k$ on $\Sigma$ so
that the cohomology classes of the 1-forms $i_{v_j}\omega$
generate $H^1_{comp}(\Sigma,\R)$. Denote by $h_i^t$ the flow of
$v_i$. Let $\cV$ be the image of the following map:
$$\begin{array}{rcl}
 (-\epsilon,\epsilon)^k  & \to & \cD\\
(t_1,\ldots,t_k) & \mapsto & \prod_{i=1}^k h_i^{t_i}.\\
\end{array}$$

\noindent Using assumption
(i) and  applying Theorem~\ref{criterion} we get that the
quasi-morphism $\mu$ is bounded on a $C^0$-neighborhood, say
$\cU$, of the identity in $\ham(\Sigma)$. Thus by (ii) and the
definition of a quasi-morphism, $\mu$ is bounded on $\cU \cdot
\cV$. But the latter set is a $C^0$-neighborhood of the identity
in $\cD$. Thus $\mu$ is continuous on $\cD$ by
Proposition~\ref{prop-general-top-group}. \qed

\section{Examples of continuous quasi-morphisms}\label{examples}

In this section we prove case by case
Theorem~\ref{theorem-main-1}. The case of the disc has been
already explained in Section~\ref{sec-highdim}. This construction
generalizes verbatim to all closed surfaces of genus $0$ with {\bf
non-empty} boundary, which proves Theorem~\ref{theorem-main-1} in
this case.

When $\Sigma$ is a closed surface of genus greater than one,
Gambaudo and Ghys constructed in \cite{gg} an infinite-dimensional
space of homogeneous quasi-morphisms on the group $\cD (\Sigma)$,
satisfying the hypothesis of Theorem~\ref{criterion}. These
quasi-morphisms are defined using $1$-forms on the surface and can
be thought of as some ``quasi-fluxes". We refer to Section 6.1 of
\cite{gg} or to Section 2.5 of \cite{ghys} for a detailed
description. The fact that these quasi-morphisms extend continuously
to the identity component of the group of area-preserving
homeomorphisms of $\Sigma$ can be checked easily without appealing
to Theorem~\ref{criterion}. This was already observed in
\cite{ghys}.

In order to settle the case of surfaces of genus one, we shall
apply the criterion given by Theorem~\ref{criterion}. The quasi-morphisms that we will use were
constructed by Gambaudo and Ghys in \cite{gg}, see also \cite{py}. We
recall briefly this construction now.

The fundamental group $\pi_{1}(\T\setminus\{0\})$ of the
once-punctured torus is a free group on two generators, $a$ and
$b$, represented by a parallel and a meridian in
$\T\setminus\{0\}$. Let $\mu : \pi_{1}(\T\setminus\{0\})\to \R$ be
a homogeneous quasi-morphism. It is known that there are plenty of
such quasi-morphisms (see \cite{brooks} for instance). We will
associate to $\mu$ a homogeneous quasi-morphism $\widetilde{\mu}$
on the group $\cD (\T)$.

We  fix a base point $x_{\ast}\in \T\setminus\{0\}$. For all $v\in
\T\setminus\{0\}$ we choose a path $\alpha_{v}(t)$, $t\in [0,1]$,
in $\T\setminus\{0\}$ from $x_{\ast}$ to $v$. We assume that the
lengths of the paths $\alpha_{v}$ are uniformly bounded with
respect to a Riemannian metric defined on the compact surface
obtained by blowing-up the origin on $\T$. Consider an element
$f\in \cD (\T)$ and fix an isotopy $(f_{t})$ from the identity to
$f$. If $x$ and $y$ are distinct points in the torus, we can
consider the curve
$$f_{t}(x)-f_{t}(y)$$
in $\T\setminus\{0\}$. Its homotopy class depends only on $f$. We
close it to form a loop:
$$\alpha(f,x,y):=\alpha_{x-y} \ast (f_{t}(x)-f_{t}(y))\ast \overline{\alpha_{f(x)-f(y)}},$$
where $\overline{\alpha_{f(x)-f(y)}} (t) := \alpha_{f(x)-f(y)}
(1-t)$. We have the cocycle relation:
$$\alpha(fg,x,y)=\alpha(g,x,y)\ast \alpha(f,g(x),g(y)).$$
Define a function $u_{f}$ on $\T\times \T\setminus\Delta$
(where $\Delta$ is the diagonal) by
$u_{f}(x,y)=\mu(\alpha(f,x,y))$. From the previous relation and
the fact that $\mu$ is a quasi-morphism we deduce the relation:
$$\left| u_{fg}(x,y)-u_{g}(x,y)-u_{f}(g(x),g(y))\right| \le
C(\mu),\ \forall f,g\in \cD (\T).$$  Moreover, it is not difficult
to see that the function $u_{f}$ is measurable and bounded on
$\T\times \T\setminus\Delta$. Hence, the map
$$f \mapsto \int_{\T\times \T}u_{f}(x,y)dxdy$$ is a quasi-morphism.
We denote by $\widetilde{\mu}$ the associated homogeneous
quasi-morphism:
$$\widetilde{\mu} (f)=\underset{p\to \infty}{{\rm lim}}\frac{1}{p}\int_{\T\times \T}u_{f^{p}}(x,y)dxdy.$$

\noindent One easily check that $\widetilde{\mu}$ is linear on any $1$-parameter subgroup. The following proposition was established in \cite{py}:

\begin{proposition}
Let $f\in \ham(\T)$ be a diffeomorphism supported in a disc $D$.
Then for any homogeneous quasi-morphism $\mu :
\pi_{1}(\T\setminus\{0\})\to \R$ one has
$$\widetilde{\mu} (f)=2\mu([a,b])\cdot {\it Cal} (f),$$
where ${\it Cal} : \ham(D) \to \R$ is the Calabi homomorphism.
\end{proposition}

By Corollary~\ref{criterion-1}, we get that the quasi-morphisms
$\widetilde{\mu}$, where $\mu$ runs over the set of homogeneous
quasi-morphisms on $\pi_{1}(\T\setminus\{0\})$ which take the
value $0$ on the element $[a,b]$, are all continuous in the
$C^{0}$-topology. According to \cite{gg}, this family
spans an infinite-dimensional vector space. To complete the proof
of Theorem~\ref{theorem-main-1} for surfaces of genus $1$, we only
have to check that the diffeomorphisms which were constructed in
\cite{gg} in order to establish the existence of an arbitrary
number of linearly independent quasi-morphisms $\widetilde{\mu}$
can be chosen to be supported in any given subsurface of genus one.
But this follows easily from the construction in Section 6.2 of
\cite{gg}.

\section{Discussion and open questions}\label{open}

\subsection{Is $\cH(\D^2)$ simple? (Le Roux's work)}\label{simplicity-leroux}

Although the algebraic structure of groups of volume-preserving
homeomorphisms in dimension greater than $2$ is well-understood
\cite{fathi}, the case of area-preserving homeomorphisms of
surfaces is still mysterious. In particular, it is unknown
whether the group $\cH(\D^2)$ is simple. Some normal subgroups of
$\cH(\D^2)$ were constructed by Ghys, Oh, and more recently by Le
Roux, see \cite{bounemoura} for a survey. However, it is unknown
whether any of these normal subgroups is a proper subgroup of
$\cH(\D^2)$. In
\cite{leroux2}, Le Roux established that the simplicity of the group
$\cH(\D^2)$ is equivalent to a certain fragmentation property. Namely, he established the following result (in
the following we assume that the total area of the disc is $1$):

{\it The group $\cH(\D^2)$ is simple if and only if there exist
numbers $\rho'<\rho$ in $(0,1]$ and an integer $N$ so that the
following holds: any homeomorphism $g\in \cH(\D^2)$, whose support
is contained in a disc of area at most $\rho$, can be written as a
product of at most $N$ homeomorphisms whose supports are contained
in discs of area at most $\rho'$.}

(By a result of Fathi \cite{fathi}, cf. \cite{leroux2}, $g$ can
always be represented as such a product with some, a priori
unknown, number of factors).

\begin{remark} {\rm
One can show that the property above depends only on $\rho$ and
not of the choice of $\rho'$ smaller than $\rho$ \cite{leroux2}.}
\end{remark}

In the sequel we will denote by $G_{\varepsilon}$ the set of
homeomorphisms in $\cH(\D^2)$ whose support is contained in an
open disc of area at most $\varepsilon$. For an element $g\in
\cH(\D^2)$ we define (following \cite{bip,leroux2})
$|g|_{\varepsilon}$ as the minimal integer $n$ such that $g$ can
be written as a product of $n$ homeomorphisms of
$G_{\varepsilon}$. Any homogeneous quasi-morphism $\phi$ on
$\cH(\D^2)$ which vanishes on $G_{\varepsilon}$ gives the
following lower bound on $|\cdot |_{\varepsilon}$:
$$|g|_{\varepsilon}\ge \frac{|\phi(g)|}{C(\phi)}\;\;\;\;\; (g\in \cH(\D^2)).$$
In particular, if $\phi$ vanishes on $G_{\varepsilon}$ but not on
$G_{\varepsilon'}$ for some $\varepsilon'>\varepsilon$, then the
norm $|\cdot |_{\varepsilon}$ is unbounded on $G_{\varepsilon'}$.

If $\phi: \cH(\D^2)\to \R$ is a homogeneous quasi-morphism which
is continuous in the $C^{0}$-topology, we can define $a(\phi)$ to
be the supremum of the positive numbers $a$ satisfying the
following property: $\phi$ vanishes on $\ham(D)$ for any disc $D$
of area less or equal than $a$ (for a homogeneous quasi-morphism
which is not continuous in the $C^{0}$-topology, one can define
$a(\phi)=0$). One can think of $a(\phi)$ as the {\it scale} at
which one can detect the nontriviality of $\phi$. According to the discussion above, the existence of a
nontrivial quasi-morphism with $a(\phi)>0$ implies that the norm
$|\cdot |_{a(\phi)}$ is unbounded on the set $G_{\rho}$ (for any
$\rho
>a(\phi)$).

According to Le Roux's result, the existence of a sequence of
continuous (for the $C^{0}$-topology) homogeneous quasi-morphisms
$\phi_{n}$ on $\cH(\D^2)$ with $a(\phi_{n})\to 0$ would imply that
the group $\cH(\D^2)$ is not simple. However, for all the examples
 of quasi-morphisms on $\cH(\D^2)$ that we know (coming from the continuous quasi-morphisms described
  in Section~\ref{sec-highdim}), one has $a(\phi)\ge \frac{1}{2}$.

\subsection{Quasi-morphisms on $\SPH^2$}
\label{sect-sphere}

Consider the sphere $\SPH^2$ equipped with an area form of
total area $1$.

\medskip
\noindent \begin{question}\label{ques-1}
\begin{itemize}
\item[{(i)}] Does there exist a non-vanishing $C^0$-continuous
homogeneous quasi-morphism on $\ham(\SPH^2)$? \item[{(ii)}] If
yes, can it be made Lipschitz with respect to Hofer's metric?
\end{itemize}
\end{question}

\medskip

If the answer to the first question was negative, this would imply that the Calabi quasi-morphism constructed in \cite{entpol} is unique. Indeed, the difference of two Calabi quasi-morphisms is continuous in the $C^{0}$-topology according to Theorem~\ref{criterion}. Note that for surfaces of positive genus, the examples of $C^{0}$-continuous quasi-morphisms that we gave are related to the existence of many Calabi quasi-morphisms \cite{py1,py}.

In turn, the affirmative answer to Question~\ref{ques-1}(ii) would
yield the solution of the following problem posed by Misha Kapovich and the second named author in 2006. It is
known \cite{pol3} that $\ham(\SPH^2)$ carries a one-parameter
subgroup, say $L:= \{f_t\}_{t \in \R}$, which is a quasi-geodesic
in the following sense: $||f_t||_H \geq c|t|$ for some $c>0$ and
all $t$. Given such a subgroup, put
$$A(L) := \sup_{\phi \in \ham(\SPH^2)} d_H (\phi,L)\;.$$

\noindent \begin{question} \label{ques-3} Is  $A(L)$ finite or
infinite?
\end{question}

The finiteness of $A(L)$ does
not depend on the specific quasi-geodesic one-parameter subgroup
$L$. Intuitively, the finiteness of $A(L)$ would yield
that the whole group $\ham(\SPH^2)$ lies in a tube of a finite
radius around $L$.

We claim that if $\ham(\SPH^2)$ admits a non-vanishing
$C^0$-continuous homogeneous quasi-morphism, which is
Lipschitz in Hofer's metric, then $A(L) = \infty$. Indeed, such a quasi-morphism would be independent from the Calabi quasi-morphism constructed in \cite{entpol}. But the existence of two independent homogeneous quasi-morphisms on $\ham(\SPH^2)$ which are Lipschitz with respect to Hofer's metric implies that $A(L)=\infty$: otherwise the finiteness of $A(L)$ would imply that Lipschitz homogeneous quasi-morphisms are determined by their restriction to $L$.

\subsection{Quasi-morphisms in higher dimensions} \label{subsec-hdpro}

Consider the following general question: given a homogeneous
quasi-morphism on $\ham(\Sigma^{2n},\omega)$, is it continuous in
$C^0$-topology?

The answer is positive, for instance, for quasi-morphisms coming
from the fundamental group $\pi_1(M)$ \cite{gg,Polt-Montreal}. It
would be interesting to explore, for instance,  the
$C^0$-continuity of a quasi-morphism $\mu$ given by the difference
of a Calabi quasi-morphism and the Calabi homomorphism
\cite{BEP,entpol} (or, more generally, by the difference of two
distinct Calabi quasi-morphisms). In order to prove the
$C^0$-continuity of $\mu$, one should establish a $C^0$-small
fragmentation lemma with a controlled number of factors in the
spirit of Lemma~\ref{lem-frag-hdim} for $\D^{2n}$ or
Theorem~\ref{lemma-N0-small-discs} for surfaces. It is likely that
the argument which we used for $\D^{2n}$ could go through without
great complications for certain Liouville symplectic manifolds,
that is compact exact symplectic manifolds which admit a
conformally symplectic vector field transversal to the boundary,
such as the open unit cotangent bundle of the sphere. 

Our result for $\D^{2n}$ should also allow the construction of continuous quasi-morphisms for groups of Hamiltonian diffeomorphisms of certain symplectic
manifolds symplectomorphic to ``sufficiently large" open subsets
of $\D^{2n}$ (for instance, the open unit cotangent bundle of a torus).

The $C^0$-small fragmentation problem on general
higher-dimen\-sio\-nal manifolds looks very difficult. Consider,
for instance, the following toy case: find a fragmentation with a
controlled number of factors for a $C^0$-small Hamiltonian
diffeomorphism supported in a sufficiently small ball $D \subset
\Sigma$. A crucial difference from the situation described in
Section~\ref{sec-highdim} is that we have no information about the
Hamiltonian isotopy $\{f_t\}$ joining $f$ with the identity: it
can ``travel" far away from $D$. In particular, when $\dim \Sigma
\geq 6$, we do not know whether $f$ lies in $\ham(D)$ or not. When
$\dim \Sigma=4$, the fact that $f\in \ham(D)$ (and hence the
fragmentation in our toy example) follows from a deep theorem by
Gromov based on pseudo-holomorphic curves techniques
\cite{Gromov}. It would be interesting to apply powerful methods
of four-dimensional symplectic topology to the $C^0$-small
fragmentation problem.

\section{Proof of the fragmentation
theorem}\label{sec-fragmentation}

In this section we prove Theorem~\ref{lemma-N0-small-discs}.
First, we need to remind a few classical results.

\subsection{Preliminaries}

In the course of the proof we will repeatedly use the following
result:

\begin{prop}
\label{prop-Moser}

Let $\Sigma$ be a compact connected oriented surface, possibly
with a non-empty boundary $\partial \Sigma$, and let $\omega_1$,
$\omega_2$ be two area-forms on $\Sigma$. Assume that $\int_\Sigma
\omega_1 = \int_\Sigma \omega_2 $. If $\partial\Sigma \neq
\emptyset$, we also assume that the forms $\omega_1$ and
$\omega_2$ coincide on $\partial\Sigma$.

Then there exists a diffeomorphism $f: \Sigma\to \Sigma$, isotopic
to the identity, such that $f^* \omega_2 = \omega_1$. Moreover,
$f$ can be chosen to satisfy the following properties:

\medskip \noindent (i) If $\partial\Sigma \neq \emptyset$, then $f$
is the identity on $\partial\Sigma$, and if $\omega_1$ and
$\omega_2$ coincide near $\partial\Sigma$, then $f$ is the
identity near $\partial\Sigma$.

\medskip \noindent (ii) If $\Sigma$ is partitioned into
polygons (with piecewise smooth boundaries), so that
$\omega_2-\omega_1$ is zero on the 1-skeleton $\Gamma$ of the
partition and the integrals of $\omega_1$ and $\omega_2$ over each
polygon are equal, then $f$ can be chosen to be the identity on
$\Gamma$.

\medskip \noindent (iii) The diffeomorphism $f$ can be chosen
arbitrarily $C^0$-close to $\identity$, provided $\omega_1$ and
$\omega_2$ are sufficiently $C^0$-close to each other (i.e.
$\omega_2 = \chi \omega_1$ for a function $\chi$ sufficiently
$C^0$-close to $1$).

\end{prop}

The existence of $f$ in the case of a closed surface follows from
a well-known theorem of Moser \cite{Moser}. The method of the
proof (``Moser's method") can be outlined as follows. Set
$\omega_t:=\omega_1 + t(\omega_2-\omega_1)$ and note that the form
$\omega_2-\omega_1$ is exact. Choose a 1-form $\sigma$ so that
$d\sigma=\omega_2-\omega_1$ and define $f$ as the time-1 flow of
the vector field $\omega_t$-dual to $\sigma$. In order to show (i)
and (ii) one has to choose a primitive $\sigma$ for
$\omega_2-\omega_1$ that vanishes near $\partial\Sigma$ or,
respectively, on $\Gamma$ -- the construction of such a $\sigma$
can be easily extracted from \cite{Banyaga-boundary}. Property
(iii) is essentially contained in \cite{Moser}: it follows easily
from the above construction of $f$, provided we can construct a
$C^0$-small primitive $\sigma$ for a $C^0$-small exact 2-form
$\omega_2-\omega_1$, but, by Lemma 1 from \cite{Moser}, it
suffices to do it on a rectangle and in this case $\sigma$ can be
constructed explicitly.

In fact, a stronger result than (iii) is true. It is known, see
\cite{oh,sikorav}, that $f$ can be chosen $C^{0}$-close to the
identity as soon as the two area forms (considered as measures)
are close in the weak-$\ast$ topology. Note that if one of the two
forms is the image of the other by a diffeomorphism $C^{0}$-close
to the identity, the two forms are close in the weak-$\ast$
topology. However, to keep this text self-contained, we are not
going to use this fact, but will reprove directly the particular
cases we need.

We equip the surface $\Sigma$ with a fixed Riemannian metric and
denote by $d$ the corresponding distance. For any map
$f: X \to \Sigma$ (where $X$ is a closed subset of $\Sigma$) we
denote by $\| f\| := \max_{x} d (x, f(x))$ its $C^0$-norm.
Accordingly, the $C^0$-norm of a smooth function $u$ defined on a
closed subset of $\Sigma$ will be denoted by $\|u\|$.

The following lemmas are the main tools for the proof.

\begin{lemma}[Area-preserving extension lemma for discs]
\label{lemma-discs}

Let $D_1\subset D_2 \subset D\subset \R^2$ be closed discs such
that $D_1 \subset {\rm Interior}\ (D_2)\subset D_2\subset {\rm
Interior}\ (D)$. Let $\phi: D_2\to D$ be a smooth area-preserving
embedding (we assume $D$ is equipped with some area form). Then
there exists $\psi\in \ham (D)$ such that
$$ \left. \psi\right|_{D_1} = \phi \;\;\;\; {\it and} \;\;\;\; \| \psi \| \to 0 \;\; {\it as}\;\; \|\phi\|\to 0.$$
\end{lemma}

\begin{lemma}[Area-preserving extension lemma for rectangles]
\label{lemma-rectangles}

Let $\Pi = [0,R]\times [-c,c]$ be a rectangle and let
$\Pi_1\subset \Pi_2\subset \Pi$ be two smaller rectangles of the
form $\Pi_i = [0,R]\times [-c_i, c_i]$ ($i=1,2$), $0< c_1<c_2<c$.
Let $\phi: \Pi_2\to \Pi$ be an area-preserving embedding (we
assume $\Pi$ is equipped with some area form) such that
\begin{itemize}

\item{} $\phi$ is the identity near $0\times [-c_2, c_2]$ and $R\times
[-c_2, c_2]$.

\item{} The area in $\Pi$ bounded by the curve $[0,R]\times y$ and
its image under $\phi$ is zero for some (and hence for all) $y\in
[-c_2, c_2]$.
\end{itemize}

\noindent Then there exists $\psi\in \ham (\Pi)$ such that
$$\left. \psi \right|_{\Pi_1} = \phi\;\;\;\; {\it and} \;\;\;\;  ||\psi ||\to 0 \;\; {\rm as}\;\; ||\phi||\to 0.$$
\end{lemma}

The lemmas will be proved in Section~\ref{sec-extension-lemmas}.

\subsection{Construction of the fragmentation}

We are now ready to prove the fragmentation theorem. In the case
when $\Sigma$ is the closed unit disc $\D^2$ in $\R^{2}$ the
theorem has been proved by Le Roux \cite{leroux2} (Proposition
4.2). In general, our proof relies on the case of the disc.

For any $b>0$ we fix a neighborhood $\mathscr{U}_{0}(b)$ of the
identity in $\ham(\D^2)$ and an integer $N_{0}(b)$ such that every
element of $\mathscr{U}_{0}(b)$ is a product of at most $N_{0}(b)$
diffeomorphisms supported in discs of area at most $b$. We will prove the following assertion.

\medskip

{\it

For any $\epsilon>0$ there exists a neighborhood $\mathscr{V}
(\epsilon)$ of the identity in $\ham(\Sigma)$, an integer $N_{1}
(\epsilon)$ and $N_{1}(\epsilon)$ discs $(D_{j})_{1\le j \le
N_{1}(\epsilon)}$ in $\Sigma$ such that any diffeomorphism $f \in
\mathscr{V} (\epsilon)$ can be written as a product $f=g_{1}\cdot
\ldots \cdot g_{N_{1}(\epsilon)}$, where each $g_{i}$ belongs to
$\ham (D_j)$ for one of the discs $D_{j}$ and is $\epsilon$-close
to the identity. ($\ast$)}

\medskip

Note that there is no restriction in $(\ast)$ on the areas of the discs
$D_j$. Let us explain how to conclude the proof of
Theorem~\ref{lemma-N0-small-discs} from this assertion. Fix $a>0$.
We can choose, for each $i$ between $1$ and $N_{1}(\epsilon)$, a
conformally symplectic diffeomorphism $\psi_{i} : \D^2 \to D_{i}$,
so that the pull-back of the area form on $\Sigma$ by $\psi_{i}$
equals the standard area form on the disc $\D^2$ times some
constant $\lambda_i>0$. If $\epsilon$ is sufficiently small, $\psi_{i}^{-1}g_{i}\psi_{i}$ is in $\mathscr{U}_{0}(\frac{a}{\lambda_{i}})$ for each $i$ and we can apply the
result for the disc to it. This
concludes the proof.

\begin{remark}
{\rm It is important that the discs $D_{i}$ as well as the maps
$\psi_{i}$ are chosen in advance, since we need the neighborhoods
$\psi_{i}\mathscr{U}_{0}(\frac{a}{\lambda_{i}})\psi_{i}^{-1} $ to
be known in advance. They determine the
neighborhood $\mathscr{V} (\epsilon)$.}
\end{remark}

We now prove $(\ast)$. The arguments we use are inspired from the
work of Fathi \cite{fathi}. Fix $\epsilon >0$. We distinguish
between two cases: 1) $\Sigma$ has a boundary, and 2) $\Sigma$ is
closed.

\medskip

\noindent {\bf First case.} Any compact connected surface with
non-empty boundary can be obtained by gluing finitely many $1$-handles to a
disc. We prove the statement $(\ast)$ by induction on the number
of $1$-handles. We already know that $(\ast)$ is true for a
disc (just take $N_1(\epsilon)=1$ and let $D_1$ be the whole
disc). Assume now that $(\ast)$ holds for any compact surface with
boundary obtained by gluing $l$ $1$-handles to the disc. Let
$\Sigma$ be a compact surface obtained by gluing a $1$-handle to a
compact surface $\Sigma_{0}$, where $\Sigma_{0}$ is obtained from
the disc by gluing $l$ $1$-handles.

Choose a diffeomorphism (singular at the corners) $\varphi : [-1,1]^{2} \to \overline{\Sigma
-\Sigma_{0}}$, sending $[-1,1]\times \{-1,1\}$ into the boundary
of $\Sigma_{0}$. Let $\Pi_{r}=\varphi([-1,1]\times [-r,r])$. Let $\mathscr{V}_1
(\epsilon)$ be the neighborhood of the identity in $\ham
(\Sigma_1)$, given by $(\ast)$ applied to the surface
$\Sigma_{1}:=\Sigma_{0}\cup \varphi([0,1]\times \{s, \vert s \vert
\ge \frac{1}{4}\})$, and let $N_{1}(\epsilon)$ be the
corresponding integer.

Let $f\in \ham (\Sigma)$, $\|f\|<\epsilon$. We apply
Lemma~\ref{lemma-rectangles} to the chain of rectangles
$\Pi_{\frac{1}{2}}\subset \Pi_{\frac{3}{4}}\subset \Pi_{\frac{7}{8}}$ and to
the restriction of $f$ to $\Pi_{\frac{3}{4}}$ (the hypothesis on
the curve $[-1,1]\times \{y\}$ is met because $f$ is Hamiltonian).
We obtain a diffeomorphism $\psi$ supported in $\Pi_{\frac{7}{8}}$ and
$C^0$-close to the identity, which coincides with $f$ on
$\Pi_{\frac{1}{2}}$. Hence, we can write
$$f=\psi h,$$
where $h$ is supported in $\Sigma_{1}$.  Since $f\in\ham(\Sigma)$
and $\psi\in\ham(\Pi_{\frac{7}{8}})$, we get that $h$ is Hamiltonian in
$\Sigma$. Since $H^{1}_{comp}(\Sigma_{1},\R)$ embeds in
$H^{1}_{comp}(\Sigma,\R)$, it means that $h$ actually belongs to
$\ham (\Sigma_1)$.

Define a neighborhood $\mathscr{V} (\epsilon)$ of the identity in
$\ham (\Sigma)$ by the following condition: $f\in \mathscr{V}
(\epsilon)$, if, first, $\|\psi\| < \epsilon$ (recall that when
$f$ converges to the identity, so does $\psi$) and, second, $h\in
\mathscr{V}_1 (\epsilon)$. Hence, if $f\in \mathscr{V}
(\epsilon)$, we can write it as a product of $N_1 (\epsilon) +1$
diffeomorphisms $g_i$, where each $g_i$ is $\epsilon$-close to the
identity and belongs to $\ham (D_j)$ for some disc $D_j\subset
\Sigma$. This proves the claim $(\ast)$ for $\Sigma$ in the first
case.

\medskip

\noindent {\bf Second case.} The surface $\Sigma$ is closed -- we
view it as a result of gluing a disc to a surface $\Sigma_{0}$
with one boundary component. Choose a diffeomorphism $\varphi :
\D^2 \to \overline{\Sigma -\Sigma_{0}}$ sending the boundary of
$\D^2$ into the boundary of $\Sigma_{0}$. Denote by $D_r$ the image by $\varphi$ of the disc of radius $r\in
[0,1]$ in $\D^2$. Let $\mathscr{V}_{1}(\epsilon)$ be the neighborhood
of the identity given by $(\ast)$ applied to the surface
$\Sigma_{1}:=\Sigma_{0}\cup \varphi(\{z\in \D^2, \vert z \vert \ge
\frac{1}{4}\})$ and let $N_{1}(\epsilon)$ be the corresponding
integer -- recall that in the first case above we have already
proved $(\ast)$ for $\Sigma_1$, which is a surface with boundary.

Let $f\in \ham (\Sigma)$, $\|f\|<\epsilon$. We apply
Lemma~\ref{lemma-discs} to the chain of discs
$D_{\frac{1}{2}}\subset D_{\frac{3}{4}}\subset D_{1}$ and to the
restriction of $f$ to $D_{\frac{3}{4}}$. We
obtain a diffeomorphism $\psi$ supported in $D_{1}$ and close to
the identity which coincides with $f$ on $D_{\frac{1}{2}}$. Hence,
we can write
$$f=\psi h,$$
where $h$ is supported in $\Sigma_{1}$. Since $f\in\ham(\Sigma)$
and $\psi\in\ham(D_1)$, we get that $h$ is Hamiltonian in
$\Sigma$. Since $\Sigma_{1}$ has one boundary component,
$H^{1}_{comp}(\Sigma_{1},\R)$ embeds in $H^{1}_{comp}(\Sigma,\R)$,
so $h$ actually belongs to $\ham (\Sigma_1)$. One concludes the proof as in the first case.

This finishes the proof of Theorem~\ref{lemma-N0-small-discs}
(modulo the proofs of the extension lemmas).

\subsection{Extension lemmas}
\label{sec-extension-lemmas}

The area-preserving lemmas for discs and rectangles will follow
from the following:

\begin{lemma}[Area-preserving extension lemma for annuli]
\label{lemma-annuli}

Let $\A = S^1 \times [-3,3]$ be a closed annulus and let $\A_1 =
S^1\times [-1,1], \A_2= S^1\times [-2,2]$ be smaller annuli inside
$\A$. Let $\phi$ be an area-preserving embedding of a fixed open
neighborhood of $\A_1$ into $\A_2$ (we assume that $\A$ is
equipped with some area form $\omega$), so that for some $y\in
[-1,1]$ (and hence for all of them) the curves $S^1\times y$ and
$\phi(S^1\times y)$ are homotopic in $\A$ and
\begin{equation}
\label{eqn-area-condition} {\it the\ area\ in}\ \A \ {\it bounded\
by}\ S^1\times y\ {\it and}\ \phi(S^1\times y)\ {\it is}\ 0.
\end{equation}

\noindent Then, there exists $\psi\in \ham (\A)$ such that $\left.
\psi\right|_{\A_1} = \phi$ and $\|\psi\|\to 0$ as $\|\phi\|\to 0$.

Moreover, if for some arc $I\subset S^1$ we have that $\phi = \identity$
outside a quadrilateral $I\times [-1,1]$ and $\phi (I\times
[-1,1])\subset I\times [-2,2]$, then $\psi$ can be chosen to be
the identity outside $I\times [-3,3]$.

\end{lemma}

Let us show how this lemma implies the area-preserving extension
lemmas for discs and rectangles.

\bigskip
\noindent {\bf Proof of Lemma~\ref{lemma-discs}.}

Up to replacing $D_2$ by a slightly smaller disc, we can assume
that $\phi$ is defined in a neighborhood of $D_2$. Identify some
small neighborhood of $\partial D_2$ with $\A = S^1 \times [-3,3]$
so that $\partial D_2$ is identified with $S^1\times
0\subset\A_1\subset\A_2\subset\A$ and $\phi (\A_1)\subset {\rm
Interior}\ (\A_2)\subset \A\subset {\rm Interior}\ (D)\setminus
\phi(D_1)$.

Apply Lemma~\ref{lemma-annuli} and find $h\in \ham (\A)$, $\|
h\|\to 0$ as $\epsilon\to 0$, so that $\left. h\right|_{\A_1} =
\phi$. Set $\phi_1 := h^{-1}\circ \phi\in \ham (D)$. Note that
$\left.\phi_1\right|_{D_1} = \phi$ and $\phi_1$ is the identity on
$\A_1$. Therefore we can extend
$\left.\phi_1\right|_{D_2\cup\A_1}$ to $D$ by the identity and get
the required $\psi$. \Qed

\bigskip
\noindent {\bf Proof of Lemma~\ref{lemma-rectangles}.}

Identify the rectangles $\Pi_1\subset \Pi_2\subset \Pi$ -- by a
diffeomorphism -- with quadrilaterals $I\times [-1,1]\subset
I\times [-2,2]\subset I\times [-3,3]$ in the annulus $\A=S^1\times
[-3,3]$ for some suitable arc $I\subset S^1$ and apply
Lemma~\ref{lemma-annuli}.\Qed

In order to prove Lemma~\ref{lemma-annuli}, we first need to prove
a version of the lemma concerning smooth (not necessarily
area-preserving) embeddings.

\begin{lemma}[Smooth extension lemma]

\label{lemma-smooth}

Let $\A_1\subset \A_2\subset\A$ be as in
Lem\-ma~\ref{lemma-annuli}. Let $\phi$ be a smooth embedding of a fixed
open neighborhood of $\A_1$ into $\A_2$, isotopic to the identity,
such that $\|\phi\| \leq \epsilon$ for some $\epsilon>0$. Then
there exists $\psi\in \dif (\A)$ such that $\psi$ is supported in
$\A_2$, $\left. \psi\right|_{\A_1} = \phi$, and $\|\psi\|\leq
C\epsilon$, for some $C>0$, independent of $\phi$.

 Moreover, if $\phi = \identity$ outside a quadrilateral $I\times [-1,1]$
and $\phi (I\times [-1,1])\subset I\times [-2,2]$ for some arc
$I\subset S^1$, then $\psi$ can be chosen to be the identity
outside $I\times [-3,3]$.

\end{lemma}

Lemma~\ref{lemma-smooth} will be proved in
Section~\ref{sec-smooth-extension-lemma-pf}.

\bigskip
\noindent {\bf Proof of Lemma~\ref{lemma-annuli}.}

As one can easily check using Proposition~\ref{prop-Moser}, we can
assume without loss of generality that the area form on
$\A=S^1\times [-3,3]$ is $\omega=dx\wedge dy$, where $x$ is the
angular coordinate along $S^1$ and $y$ is the coordinate along
$[-3,3]$. All norms and distances are measured with the Euclidean
metric on $\A$.

Denote $\A_+ := S^1\times [1,2]$, $\A_- := S^1\times [-2, -1]$.

Assume $\|\phi\| < \epsilon$. By Lemma~\ref{lemma-smooth}, there
exists $f\in \dif (\A_2)$ such that $\| f\|\leq C\epsilon$, and $f
= \phi$ on a neighborhood of $\A_1$. Denote $\Omega := f^*
\omega$. By \eqref{eqn-area-condition},
\begin{equation}
\label{eqn-areas-Omega-omega-equal} \int_{\A_+} \Omega =
\int_{\A_+} \omega,\ \int_{\A_-} \Omega = \int_{\A_-} \omega.
\end{equation}
Note that $\Omega$ coincides with $\omega$ on a neighborhood of
$\partial \A_+$ and $\partial\A_-$. Let us find $h\in \dif (\A_2)$
such that
\begin{itemize}

\item{} $\left. h\right|_{\A_1} = \identity$,

\item{} $h^* \Omega = \omega$,

\item{} $\| h\|\to 0$ as $\epsilon\to 0$.

\end{itemize}

Given such an $h$, we extend $f h$ by the identity to the whole of
$\A$. The resulting diffeomorphism of $\A$ is $C^0$-small (if
$\epsilon$ is sufficiently small), preserves $\omega$ and belongs
to $\dif(\A)$, hence (see e.g. \cite{Tsuboi}), also to $\cD (\A)$.
It may not be Hamiltonian but one can easily make it Hamiltonian
by a $C^0$-small adjustment on $\A\setminus \A_2$. The resulting
diffeomorphism $\psi \in \ham (\A)$ will have all the required
properties.

\bigskip
\noindent {\bf Preparations for the construction of $h$.}

Since on $\A_1$ the map $h$ is required to be identity, we need to
construct it on $\A_+$ and $\A_-$. We will construct $h_+:= \left.
h\right|_{\A_+}$, the case of $\A_-$ is similar. By a rectangle or
a square in $\A$ we mean the product of a connected arc in $S^1$
and an interval in $[-3,3]$.

Let us divide $\A_+ = S^1\times [1,2]$ into closed squares
$K_1,\ldots, K_N$, with a side of size
$r=\epsilon^{1/4}>3\epsilon$ (we assume that $\epsilon$ is
sufficiently small). Denote by $V$ the set of the vertices which
are not on the boundary and by $E$ the set of the edges which are
not on the boundary. Finally, denote by $\Gamma$ the 1-skeleton of
the partition (i.e. the union of all the edges).

For each $v\in V$ denote by $B_v (\delta)$ the open ball in $\A_+$
of radius $\delta>0$ with the center at $v$. Fix a small positive
$\delta_0 < r$ so that for $0<\delta<\delta_0$, the balls $B_v
(\delta)$, $v\in V$, are disjoint and each $B_v (\delta)$
intersects only the edges adjacent to $v$. Given such a $\delta$,
consider for each edge $e\in E$  a small open rectangle $U_e
(\delta)$ covering $e\setminus \big( e\cap \cup_{v\in V} B_v
(\delta)\big)$, so that
\begin{itemize}

\item{} $U_e (\delta) \cap B_v (\delta) \neq \emptyset$ if and
only if $v$ is adjacent to $e$.

\item{} $U_e (\delta)$ does not intersect any other edge apart
from $e$.

\item{} All the rectangles $U_e (\delta)$, $e\in E$, are mutually
disjoint.

\end{itemize}

\noindent Define a neighborhood $U(\delta)$ of $\Gamma$ by
$$U(\delta) =
\left( \cup_{v\in V} B_v (\delta) \right) \cup \left( \cup_{e\in
E} U_e (\delta)\right).$$

For each $\varepsilon_1 > \varepsilon_2>0$ we pick a cut-off
function $\chi_{\varepsilon_1, \varepsilon_2}: \R\to [0,1]$ which
is equal to $1$ on a neighborhood of $(-\varepsilon_2,
\varepsilon_2)$ and vanishes outside $(-\varepsilon_1,
\varepsilon_1)$. Finally, by $C_1, C_2,\ldots$ we will denote
positive constants independent of $\epsilon$.

The construction of $h_+$ will proceed in several steps.

\bigskip
\noindent {\bf Adjusting $\Omega$ on $\Gamma$.}

We are going to adjust the form $\Omega$ by a diffeomorphism
supported inside $U (\delta)$ to make it
equal to $\omega$ on $\Gamma$. One can first construct $h_1\in \dif (\A_+)$ supported in
$\cup_{v\in V} B_{v}(2\delta)$ such that $h_1^* \Omega = \omega$
on $\cup_{v\in V} B_v (\delta)$ for some $\delta<\delta_0$ (simply
using Darboux charts for $\Omega$ and $\omega$). Note that $\|
h_1\| < 2\delta$. Write $\Omega':= h_1^* \Omega$. For each $e\in
E$ we will construct a diffeomorphism $h_e$ supported in $U_e
(\delta)$ so that $h_e^* \Omega' = \omega$ on $l:=U_e (\delta)\cap
e$ (and thus on the whole $e$, since $\Omega'$ already equals
$\omega$ on each $B_v (\delta)$).

Without loss of generality, let us assume that $e$ does not lie on
$\partial \A_+$ (since $\Omega'$ already coincides with $\omega$
there) and that $U_e (\delta)$ is of the form $(a,b)\times
(-\delta, \delta)$. Write the restriction of $\Omega'$ on
$l=(a,b)\times 0$ as $\beta (x) dx\wedge dy$, $\beta (x) >0$.

Consider a cut-off function $\chi=\chi_{\delta,\delta/2}: \R\to
[0,1]$ and define a vector field ${\bf w} (x,y)$ on $U_e (\delta)$
by
$${\bf w}(x,y)= \chi (y) \log (\beta(x)) y \frac{\partial}{\partial y}.$$
Note that ${\bf w}=0$ on $l$ and has compact support in $U_e
(\delta)$ (the end-points of $l$ lie in the balls $B_v (\delta)$
on which $\Omega=\omega$ and thus $\beta=1$ near these endpoints).
Let $\varphi_t$ be the flow of ${\bf w}$. A simple calculation
shows that
$$\frac{d}{dt}\varphi_{t}^{\ast}\omega=\varphi_{t}^{\ast}L_{{\bf w}}\omega=\log(\beta(x))e^{t\log(\beta(x))}dx\wedge dy$$
at the point $\varphi_{t}((x,0))=(x,0)$. Therefore $\varphi_{1}^{*}\omega=\Omega'$ on $l$. Thus setting
$h_e := \varphi_1^{-1}$ we get that $h_e^* \Omega' = \omega$ on
$l$ and that $\| h_e\|\leq 2\delta$, because $h_e$ preserves the
fibers $x\times (-\delta, \delta)$. Set
$$h_2 := \prod_{e\in E} h_e.$$
Since the rectangles $U_e (\delta)$ are
pairwise disjoint, $h_2$ is supported in $U (\delta)$ and satisfies the conditions

\begin{itemize}
\item{} $h_2^* \Omega' = \omega$ on $\Gamma$.
\item{} $\|h_2\|\leq 2\delta$.
\end{itemize}

\noindent The diffeomorphism $h_3:=h_{2}h_{1}\in \dif(\A_+)$
satisfies $\|h_{3}\|\le 4\delta$ and
$$h_3^* \Omega = h_2^*\Omega' = \omega\ {\rm on}\ \Gamma.$$

\noindent Consider the area form $\Omega'':= h_3^* \Omega$. It
coincides with $\omega$ on the 1-skeleton $\Gamma$ and near
$\partial \A_+$. Moreover, $\int_{\A_+} \Omega'' = \int_{\A_+}
\Omega'$ and hence, by \eqref{eqn-areas-Omega-omega-equal},
\begin{equation}
\label{eqn-Omega-double-prime-omega}
\int_{\A_+} \Omega'' =
\int_{\A_+} \omega.
\end{equation}

\bigskip
\noindent {\bf Adjusting the areas of the squares}

In this paragraph we construct a $C^{0}$-perturbation $\rho
\omega$ of $\omega$ which has the same integral as $\Omega''$ on
each square $K_{i}$.

Making $\delta$ sufficiently small we can assume that $\|h_3\|
<\epsilon$. Recall that $r=\epsilon^{1/4}>3\epsilon$. Therefore
the image of one of the squares $K_{i}$ by $h_{3}$ contains a
square of area $(r-\epsilon)^{2}$ and is contained in a square of
area $(r+\epsilon)^{2}$. Hence,
$$\frac{(r-2\epsilon)^2}{r^2}\leq \frac{\int_{K_{i}} \Omega''}{\int_{K_{i}}
\omega} \leq \frac{(r+2\epsilon)^2}{r^2}.$$ Since $\epsilon/r =
\epsilon^{3/4}\to 0$ as $\epsilon\to 0$, we get that if $\epsilon$
is sufficiently small, there exists $C_1>0$ so that
\begin{equation}
\label{eqn-areas-Omega-double-prime-arbitrary-squares}
1-C_1\frac{\epsilon}{r}\leq \frac{\int_{K_{i}}
\Omega''}{\int_{K_{i}} \omega} \leq 1+ C_1\frac{\epsilon}{r}.
\end{equation}
Now set $s_i:=\int_{K_i} \Omega''$ and $t_i = s_i/r^2 -1$. By
\eqref{eqn-areas-Omega-double-prime-arbitrary-squares},
\begin{equation}
\label{eqn-t-i} | t_i|\leq C_1 \frac{\epsilon}{r} = C_1
\epsilon^{3/4}.
\end{equation}

For each $i$ we can choose a nonnegative function $\bar{\rho}_{i}$
supported in the interior of $K_{i}$ so that
$\int_{K_{i}}\bar{\rho}_{i}\omega=r^{2}$ and
\begin{equation}
\label{eqn-rho}
 ||\bar{\rho}_{i}||_{C^{2}}\le
C_{2}\epsilon^{-1/2} \end{equation} for some constant $C_{2}>0$
independent of $i$. Define a function $\varrho$ on $\A$ by
$$\varrho := 1+\sum_{i=1}^N t_i \bar{\rho}_i.$$
By \eqref{eqn-t-i} and \eqref{eqn-rho}, the function $\varrho$ is
positive and the form $\varrho \omega$ converges to $\omega$ (in
the $C^{0}$-sense) as $\epsilon$ goes to $0$.  Moreover, $\varrho$
is equal to $1$ on $\Gamma$ and the two area forms $\varrho
\omega$ and $\Omega''$ have the same integral on each $K_{i}$. By \eqref{eqn-Omega-double-prime-omega}, one has:
\begin{equation}\label{eqn-varrho-omega-omega} 
\int_{\A_+} \varrho\omega = \int_{\A_+}\Omega''=
\int_{\A_+} \omega.
\end{equation}

\bigskip
\noindent {\bf Finishing the construction of $h_+$: Moser's
argument.}

Let us apply Proposition~\ref{prop-Moser}, part (ii), to the forms
$\Omega''$ and $\varrho\omega$ on $\A_+$: these forms have the same integral over each $K_i$ and
coincide on $\Gamma$ and near the boundary of $\A_+$, therefore
there exists a diffeomorphism $h_4\in\dif (\A_+)$ which is the
identity on $\Gamma$ and satisfies $h_4^* \Omega'' = \varrho
\omega$. Since $h_4$ is the identity on $\Gamma$ and maps each
$K_i$ into itself, its $C^0$-norm is bounded by the diameter of
$K_i$, hence goes to $0$ with $\epsilon$.

Finally, apply Proposition~\ref{prop-Moser} to the forms $\omega$
and $\varrho\omega$ on $\A_+$: by \eqref{eqn-varrho-omega-omega},
their integrals over $\A_+$ are the same, they coincide on
$\partial\A_+$ and are $C^0$-close. Therefore there exists $h_5\in
\dif (\A_+)$ so that $h_5^* (\varrho\omega) = \omega$ and
\begin{equation}
\label{eqn-h5} \| h_5\|\to 0\ {\rm as}\ \epsilon\to 0.
\end{equation}
Then $h_+ := h_5 h_4 h_3$ is the required
diffeomorphism. This finishes the construction of $h$.

\bigskip
\noindent {\bf Final observation.}

Note that if $\phi = \identity$ outside a quadrilateral $I\times [-1,1]$
for some arc $I\subset S^1$, then $f$ can be chosen to have the
same property. In such a case we need to construct $h_+\in \dif
(\A_+)$ supported in $I\times [-3,3]$.

Let $J$ be the complement of the interval $I$ in the circle. The
partition of $\A_+$ into squares can be chosen so that it extends
a partition of $J\times [1,2]\subset \A_+$ into squares of the
same size. Going over each step of the construction of $h_+$
above, we see that, since $\Omega = \omega$ on $J\times [1,2]$,
each of the maps $h_1, h_2, h_3, h_4, h_5$ can be chosen to be
identity on each of the squares in $J\times [1,2]$, hence on the
whole $J\times [1,2]$. Therefore $h_+$, hence $h$, hence $\psi=fh$,
is the identity on $J\times [1,2]$. Moreover $\psi$ is automatically
Hamiltonian in this case.\Qed

\subsection{Proof of the smooth extension lemma}
\label{sec-smooth-extension-lemma-pf}

As in the proof of Lemma~\ref{lemma-annuli}, we assume that the
Riemannian metric on $\A=S^1\times [-3,3]$ used for the
measurements is the Euclidean product metric. We can also assume
that the neighborhood of $\A_1$ on which $\phi$ is defined is, in
fact, an open neighborhood of $\A':=S^1\times [-1.5, 1.5]$ and
that $\epsilon \ll 0.5$.

\bigskip
\noindent {\bf Proof of Lemma~\ref{lemma-smooth}.}

Applying Lemma~\ref{lemma-curves-khanevsky} (see the appendix by
M. Khanevsky below) to the two curves $S^1\times \{\pm 1.5\}$ and
their images under $\phi$ we can find $\psi_1\in \dif (\A)$,
supported in $S^1\times (-2,-1) \cup S^1\times (1,2)$, such that
$\psi_1$ coincides with $\phi^{-1}$ on the curves $\phi
(S^1\times \{\pm 1.5\})$. Moreover it satisfies $\|\psi_1\| <
C'\epsilon$. Define $\psi_2:=\psi_1\phi$. This map is defined on an open neighborhood of $\A'=S^1\times
[-1.5, 1.5]$ and has the following properties:

\begin{itemize}

\item{} The restriction of $\psi_2$ to $\A'$ is a diffeomorphism
of $\A'$. It is the identity on $\partial\A'$ and coincides with
$\phi$ on $\A_1 = S^1\times [-1,1]\subset \A'$.

\item{} $\|\psi_2\| < C''\epsilon$, where $C'':= C'+1$.

\end{itemize}

We are going to modify $\psi_2$ (by a $C^0$-small perturbation) to
make it the identity not only on $\partial \A'$ but on {\it an
open neighborhood of $\partial\A'$}. Then we will extend it by the
identity to a diffeomorphism of $\A$ with the required properties.

Since $\psi_2$ is the identity on $\partial \A'$, by
perturbing it slightly near $\partial\A'$ (in the $C^\infty$-norm)
we can assume that, in addition to the properties listed above,
near $\partial\A'$ the map $\psi_2$ preserves the foliation of
$\A$ by the circles $S^1\times y$. It means that for some
sufficiently small $r>0$ the restriction of $\psi_2$ to $S^1\times
[-1.5, -1.5+r] \cup S^1\times [1.5-r, 1.5]$ has the form
$$\psi_2: (x,y)\mapsto (x+ u(x,y), y),$$
for some smooth function $u$ such that $\| u\| < C''\epsilon$.
Choose a cut-off function $\chi= \chi_{1.5, 1.5-r}:\R\to [0,1]$
and define a map $\psi_3$ on $\A'$ as follows:
$$\psi_3:=\psi_2\ {\rm on}\ S^1\times [-1.5+r,
1.5-r],$$
$$\psi_3 (x,y):= (x+ \chi(y) u(x,y), y),\ {\rm when}\ |y|\geq
1.5-r.$$

\noindent We now consider the diffeomorphism $\psi$ which equals
$\psi_{3}$ on $\A'$ and the identity outside $\A'$. It coincides
with $\phi$ on $\A_1$ and satisfies $\|\psi\| < C''\epsilon$.
Note that if $\epsilon$ is sufficiently small, $\psi$
automatically belongs to the identity component $\dif (\A)$ (this
can be easily deduced, for instance, from \cite{EE1, EE2} or
\cite{Smale}). This finishes the construction of $\psi$ in the
general case.

Let us now consider the case, where $\phi = \identity$ outside a
quadrilateral $I\times [-1,1]$ and $\phi (I\times [-1,1])\subset
I\times [-2,2]$ for some arc $I\subset S^1$. Then, by
Lemma~\ref{lemma-curves-khanevsky}, we can assume that $\psi_1$ is
supported in $I\times [-3,3]$. Then $\psi_2$ is the identity
outside $I\times [-1.5, 1.5]$. When we perturb $\psi_2$ near
$\partial\A'$ to make it preserve the foliation by circles, we can
choose the perturbation to be supported in $I\times [-1.5,1.5]$.
Thus $u(x,y)$ would be $0$ outside $I\times [-1.5,1.5]$. This
yields that $\psi_3$, and consequently $\psi$, are the identity
outside $I\times [-3,3]$. \Qed

\bigskip

\section{Appendix by Michael Khanevsky:\\ An extension lemma for curves}\label{app}

\bigskip For a diffeomorphism $\phi$ of a compact surface with a
Riemannian distance $d$ we write $||\phi|| = \max d(x,\phi(x))$.
The purpose of this appendix is to prove the following extension
lemma which was used in
Section~\ref{sec-smooth-extension-lemma-pf} above.

\medskip
\noindent
\begin{lemma}
\label{lemma-curves-khanevsky}

Let $A := S^1\times [-1, 1]$ be an annulus equipped with the
Euclidean product metric. Set $L= S^1\times 0$. Assume that $\phi$
is a smooth embedding of an open neighborhood of $L$ in $A$, so
that $L$ is homotopic to $\phi (L)$ and $\|\phi\|\leq
\epsilon$ for some $\epsilon\ll 1$.

Then there exists a diffeomorphism $\psi\in \dif (A)$ such that
$\psi  = \phi$ on $L$ and $\|\psi\|< C'\epsilon$ for some
$C'>0$ independent of $\psi$.

Moreover, if $\phi = \identity$ outside some arc $I\subset L$ and $\phi
(I)\subset I\times [-1,1]$, then $\psi$ can be made the identity
outside $I\times [-1,1]$.

\end{lemma}

\medskip
\noindent\begin{proof} We view the coordinate $x$
on $A$ along $S^1$ as a horizontal one, and the coordinate $y$
along $[-1,1]$ as a vertical one. If $a,b\in L$ are not antipodal,
we denote by $[a,b]$ the shortest closed arc in $L$ between $a$
and $b$.

The proof consist of a few steps. By $C_1, C_2, \ldots$ we will
denote some universal positive constants.

\bigskip
\noindent {\bf Step 1.} Shift the curve $\phi(L)$ by $3\epsilon$ upward by a diffeomorphism $\psi_{1}\in \dif(A)$ with $\| \psi_{1} \|\leq C_1\epsilon$, so that $K:=\psi_{1}(\phi(L))$ lies strictly above $L$ (see figure~\ref{1}).

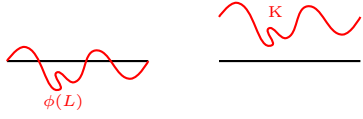
\begin{figure}\begin{center}
\psset{unit=.47cm}
\begin{pspicture}(0,.72)(11,4)
\psline[linewidth=.06](.5,2)(4.5,2)
\psline[linewidth=.06](6.5,2)(10.5,2)
\pscurve[linecolor=red,linewidth=.06](.5,2)(1,2.4)(2,1.2)(1.9,1.67)(2.4,1.4)(3,2.3)(4,1.5)(4.5,2)
\pscurve[linecolor=red,linewidth=.06](6.5,3.25)(7,3.65)(8,2.45)(7.9,2.92)(8.4,2.65)(9,3.55)(10,2.75)(10.5,3.25)
\rput{0}(2,.84){\red \begin{tiny} $\phi(L)$\end{tiny}}
\rput{0}(8,3.41){\red \begin{tiny} K\end{tiny}}
\end{pspicture}
\caption{Shifting $L$}\label{1}
\end{center}\end{figure}

\bigskip
\noindent {\bf Step 2.} Let $x_1,\ldots, x_N$ be points on $L$ chosen in a cyclic order so
that the distance between any two consecutive points $x_i$ and
$x_{i+1}$ is at most $\epsilon$ (here and further on, $i+1$ is
taken to be $1$, if $i=N$).

For each $i=1,\ldots, N$, consider a vertical ray originating at
$x_i$ and assume, without loss of generality, that it is
transversal to $K$ and that $K$ is parallel to $L$ near its
intersection points with the ray. Among the intersection points of
the ray with $K$ choose the closest one to $L$ and denote it by
$y_i$. Denote by $r_i$ the closed vertical interval between $x_i$
and $y_i$. Choose small pairwise disjoint open rectangles $U_i$,
of width at most $\epsilon/3$ and of height at most $4\epsilon$,
around each of the intervals $r_i$.

For each $i=1,\ldots, N$, it is easy to construct a diffeomorphism
$\psi_{2,i}$ supported in $U_i$ which moves a connected arc of
$K\cap U_i$ containing $y_i$ by a parallel shift downwards into an
arc of $L$ containing $x_i$ so that $\psi_{2,i} (K)$ lies
completely in $\{ y\geq 0\}$. Set $\psi_2:=\prod_{i=1}^N
\psi_{2,i}$. Clearly, $\|\psi_{2,i}\|\leq C_2\epsilon$ for each
$i$ and therefore, since the supports of all the diffeomorphisms
$\psi_{2,i}$ are disjoint, $\|\psi_2\|\leq C_2\epsilon$ as well. Set (see figure~\ref{2})
$$\tpsi: = \psi_2\psi_1\in \dif (A),\;\;\;\;\; \tK := \tpsi (\phi(L)).$$
Note that $\| \tpsi\|\leq C_3\epsilon$.

\bigskip
\noindent {\bf Step 3.} Note that the points $x_i$, $i=1,\ldots,N$, lie on $\tK$ and that
$$\tK\subset \{ y\geq 0\}.$$ An easy topological argument shows that
in such a case, since the points $x_i$ lie on $L$ in cyclic order,
they also lie in the same cyclic order on $\tK$.

For each $i$ there are two arcs in $\tK$ connecting $x_i$ and
$x_{i+1}$ -- denote by $K_i$ the one homotopic with fixed
endpoints to the arc $[x_i, x_{i+1}]\subset L$. Since the points
$x_i$ lie on $\tK$ in the same cyclic order as on $L$, we see that
$K_1,\ldots,K_N$ are precisely the closures of the $N$ open arcs
in $\tK$ obtained by removing the points $x_1,\ldots,x_N$ from
$\tK$.

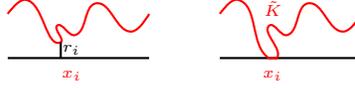
\begin{figure}\begin{center}
\psset{unit=.47cm}
\begin{pspicture}(0,.72)(10,4)
\psline[linewidth=.06](.5,2)(4.5,2)
\pscurve[linecolor=red,linewidth=.06](.5,3.25)(1,3.65)(2,2.45)(1.9,2.92)(2.4,2.65)(3,3.55)(4,2.75)(4.5,3.25)
\psline[linewidth=0.06](2,2.45)(2,2)
\rput{0}(2.31,2.24){\begin{tiny}$r_{i}$\end{tiny}}
\rput{0}(2.31,1.5){\red \begin{tiny}$x_{i}$\end{tiny}}
\psline[linewidth=.06](6.5,2)(10.5,2)
\pscurve[linecolor=red,linewidth=.06](6.5,3.25)(7,3.65)(7.89,2)(8,2)(8.11,2)(7.9,2.92)(8.4,2.65)(9,3.55)(10,2.75)(10.5,3.25)
\rput{0}(8,3.41){\red \begin{tiny}$\tilde{K}$\end{tiny}}
\rput{0}(8,1.5){\red \begin{tiny}$x_{i}$\end{tiny}}
\end{pspicture}
\caption{$\tilde{K}$ coincides with $L$ near $x_{i}$}\label{2}
\end{center}\end{figure}

Let $B_{i}$ be the open set bounded by $K_{i}$ and $[x_{i},x_{i+1}]$ (see figure~\ref{3}). The $B_{i}$'s are disjoint and have diameter at most $C_{4} \epsilon$. Let $B_{i}'$ be disjoint open neighborhoods of the $B_{i}$'s of diameter at most $C_{5}\epsilon$. Now for each $i$ one can easily find a
diffeomorphism $\psi_{3,i}\in \dif (B'_i)$ such that
$\|\psi_{3,i}\|\leq C_5\epsilon$ and $\psi_{3,i} (K_i) = [x_i,
x_{i+1}]$. Set $\psi_3:=\prod_{i=1}^N \psi_{3,i}$. Since the supports of all
$\psi_{3,i}$ are pairwise disjoint we get that $\|\psi_3\|\leq
C_5\epsilon$.

\bigskip
\noindent {\bf Step 4.} Define $\psi_4 := \psi_3\tpsi=\psi_3\psi_2\psi_1$. Clearly,
$\psi_4\in \dif (A)$ and $\|\psi_4\|\leq C_6 \epsilon$. Recall
that for each $i$ we have $\psi_3 (K_i) = [x_i, x_{i+1}]$ and that
each $K_i$ is the shortest arc between $x_i$ and $x_{i+1}$ in $\tK
= \psi_2\psi_1 (L)$. Thus $\psi_4$ maps $K$ into $L$. The diffeomorphism $\psi_{4}^{-1}$ satisfy $\psi_{4}^{-1}(L)=\phi(L)$. We now obtain easily the required $\psi$ by a $C^{0}$-small perturbation of $\psi_{4}^{-1}$.\end{proof}

\begin{figure}\begin{center}
\psset{unit=.47cm}
\begin{pspicture}(0,.72)(4,4)
\psline[linewidth=.06](.5,2)(4.5,2)
\pscurve[linecolor=red,linewidth=.06](.39,3)(.89,2)(1,2)(1.11,2)(.9,2.92)(1.4,2.65)(2,3.55)(3,2.75)(3.5,3.25)(4.01,2)(4.12,2)(4.16,2)(4.5,2.7)
\rput{0}(2.7,2.29){\red \begin{tiny} $B_{i}$ \end{tiny}}
\rput{0}(1.05,1.5){\red \begin{tiny} $x_{i}$ \end{tiny}}
\rput{0}(4.07,1.5){\red \begin{tiny} $x_{i+1}$ \end{tiny}}
\end{pspicture}
\caption{The open set $B_{i}$}\label{3}
\end{center}\end{figure}
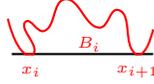

\bigskip
\begin{small}
\begin{tabular}{lll}
Michael Entov & & Leonid Polterovich\\
Department of Mathematics & & School of Mathematical Sciences\\
Technion - Israel Institute of Technology & & Tel Aviv University\\
Haifa 32000, Israel & & Tel Aviv 69978, Israel\\
entov@math.technion.ac.il & & polterov@post.tau.ac.il\\
 & & \\
Pierre Py & & Michael Khanevsky\\
Department of Mathematics & & School of Mathematical Sciences\\
University of Chicago & & Tel Aviv University\\
Chicago, Il 60637, USA & & Tel Aviv 69978, Israel\\
pierre.py@math.uchicago.edu & & khanev@post.tau.ac.il\\    
\end{tabular}
\end{small}

\end{document}